\documentclass{amsart}
\usepackage{amssymb, amsthm, amsmath, amscd, epsfig, psfrag}

\newtheorem{theorem}{Theorem}[section]
\newtheorem{lemma}[theorem]{Lemma}

\newtheorem{proposition}[theorem]{Proposition}

\theoremstyle{definition}

\newtheorem{remark}[theorem]{Remark}
\newtheorem{example}[theorem]{Example}
\newtheorem{notation}[theorem]{Notation}
\numberwithin{equation}{theorem}

\def\fraka{{\mathfrak a}}
\def\frakb{{\mathfrak b}}
\def\frakm{{\mathfrak m}}
\def\frakn{{\mathfrak n}}
\def\frakp{{\mathfrak p}}

\def\frakM{{\mathfrak M}}
\def\frakP{{\mathfrak P}}

\def\bsa{{\boldsymbol a}}

\def\bsa{{\boldsymbol a}}
\def\bsm{{\boldsymbol m}}
\def\bsn{{\boldsymbol n}}
\def\bsx{{\boldsymbol x}}
\def\bsy{{\boldsymbol y}}
\def\bsz{{\boldsymbol z}}
\def\bszero{{\boldsymbol 0}}
\def\bsgamma{{\boldsymbol\gamma}}

\def\calO{{\mathcal O}}
\def\calR{{\mathcal R}}

\def\AA{{\mathbb A}}
\def\CC{{\mathbb C}}
\def\NN{{\mathbb N}}
\def\PP{{\mathbb P}}
\def\QQ{{\mathbb Q}}
\def\ZZ{{\mathbb Z}}

\def\d{\Delta}
\def\nat{\natural}
\def\phi{\varphi}
\def\bar{\overline}
\def\hat{\widehat}

\def\seg{{\scriptstyle{\#}}}
\def\to{\longrightarrow}
\def\mapsto{\longmapsto}
\def\ge{\geqslant}
\def\le{\leqslant}

\def\coker{\operatorname{coker}}

\def\Cl{\operatorname{Cl}}
\def\gr{\operatorname{gr}}
\def\Hom{\operatorname{Hom}}

\def\Ext{\operatorname{Ext}}
\def\Proj{\operatorname{Proj}}
\def\rank{\operatorname{rank}}
\def\Spec{\operatorname{Spec}}

\def\HH{\underline{\Hom}}
\def\HE{\underline{\Ext}}

\begin{document}
\title[Multigraded rings]{Multigraded rings, diagonal subalgebras, and rational singularities}

\author{Kazuhiko Kurano}
\address{Department of Mathematics, Meiji University, Higashimita 1-1-1, Tama-ku,\newline Kawasaki-shi 214-8571, Japan}
\email{kurano@math.meiji.ac.jp} 

\author{Ei-ichi Sato}
\address{Department of Mathematics, Kyushu University, Hakozaki 6-10-1, Higashi-ku,\newline Fukuoka-city 812-8581, Japan}
\email{esato@math.kyushu-u.ac.jp}

\author{Anurag K. Singh}
\address{Department of Mathematics, University of Utah, 155 South 1400 East, \newline Salt Lake City, UT 84112, USA}
\email{singh@math.utah.edu}

\author{Kei-ichi Watanabe}
\address{Department of Mathematics, Nihon University, Sakura-Josui 3-25-40, Setagaya,\newline Tokyo 156-8550, Japan}
\email{watanabe@math.chs.nihon-u.ac.jp}

\thanks{A.K.S. was supported by NSF grants DMS 0300600 and DMS 0600819.}


\dedicatory{To Paul Roberts}

\maketitle

\section{Introduction}

We study the properties of F-rationality and F-regularity in multigraded rings and their diagonal subalgebras. The main focus is on diagonal subalgebras of bigraded rings: these constitute an interesting class of rings since they arise naturally as homogeneous coordinate rings of blow-ups of projective varieties.

Let $X$ be a projective variety over a field $K$, with homogeneous coordinate ring $A$. Let $\fraka\subset A$ be a homogeneous ideal, and $V\subset X$ the closed subvariety defined by $\fraka$. For $g$ an integer, we use $\fraka_g$ to denote the $K$-vector space consisting of homogeneous elements of $\fraka$ of degree $g$. If $g\gg0$, then $\fraka_g$ defines a very ample complete linear system on the blow-up of $X$ along $V$, and hence $K[\fraka_g]$ is a homogeneous coordinate ring for this blow-up. Since the ideals $\fraka^h$ define the same subvariety $V$, the rings $K[(\fraka^h)_g]$ are homogeneous coordinate ring for the blow-up provided $g\gg h>0$.

Suppose that $A$ is a standard $\NN$-graded $K$-algebra, and consider the $\NN^2$-grading on the Rees algebra $A[\fraka t]$, where $\deg rt^j=(i,j)$ for $r\in A_i$. The connection with diagonal subalgebras stems from the fact that if $\fraka^h$ is generated by elements of degree less than or equal to $g$, then
\[
K[(\fraka^h)_g]\cong\bigoplus_{k\ge0}{A[\fraka t]}_{(gk,hk)}\,.
\]
Using $\d=(g,h)\ZZ$ to denote the $(g,h)$-\emph{diagonal} in $\ZZ^2$, the \emph{diagonal subalgebra} $A{[\fraka t]}_\d=\oplus_k{A[\fraka t]}_{(gk,hk)}$ is a homogeneous coordinate ring for the blow-up of $\Proj A$ along the subvariety defined by $\fraka$, whenever $g\gg h>0$.

The papers \cite{GG,GGH,GGP,Trung} use diagonal subalgebras in studying blow-ups of projective space at finite sets of points. For $A$ a polynomial ring and $\fraka$ a homogeneous ideal, the ring theoretic properties of $K[\fraka_g]$ are studied by Simis, Trung, and Valla in \cite{STV} by realizing $K[\fraka_g]$ as a diagonal subalgebra of the Rees algebra $A[\fraka t]$. In particular, they determine when $K[\fraka_g]$ is Cohen-Macaulay for $\fraka$ a complete intersection ideal generated by forms of equal degree, and also for $\fraka$ the ideal of maximal minors of a generic matrix. Some of their results are extended by Conca, Herzog, Trung, and Valla as in the following theorem:

\begin{theorem}\cite[Theorem 4.6]{CHTV}
Let $K[x_1,\dots,x_m]$ be a polynomial ring over a field, and let $\fraka$ be a complete intersection ideal minimally generated by forms of degrees $d_1,\dots,d_r$. Fix positive integers $g$ and $h$ with $g/h>d=\max\{d_1,\dots,d_r\}$.

Then $K[(\fraka^h)_g]$ is Cohen-Macaulay if and only if $g>(h-1)d-m+\sum_{j=1}^r d_j$.
\end{theorem}

When $A$ is a polynomial ring and $\fraka$ an ideal for which $A[\fraka t]$ is Cohen-Macaulay, Lavila-Vidal \cite[Theorem~4.5]{Olga:MM} proved that the diagonal subalgebras $K[(\fraka^h)_g]$ are Cohen-Macaulay for $g\gg h\gg 0$, thereby settling a conjecture from \cite{CHTV}. In \cite{CH} Cutkosky and Herzog obtain affirmative answers regarding the existence of a constant $c$ such that $K[(\fraka^h)_g]$ is Cohen-Macaulay whenever $g\ge ch$. For more work on the Cohen-Macaulay and Gorenstein properties of diagonal subalgebras, see \cite{HHR, Hyry:TAMS, Olga:thesis}, and \cite{LVZ}.

As a motivating example for some of the results of this paper, consider a polynomial ring $A=K[x_1,\dots,x_m]$ and an ideal $\fraka=(z_1,z_2)$ generated by relatively prime forms $z_1$ and $z_2$ of degree $d$. Setting $\d=(d+1,1)\ZZ$, the diagonal subalgebra ${A[\fraka t]}_\d$ is a homogeneous coordinate ring for the blow-up of $\Proj A=\PP^{m-1}$ along the subvariety defined by $\fraka$. The Rees algebra $A[\fraka t]$ has a presentation
\[
\calR= K[x_1,\dots,x_m, y_1,y_2]/(y_2z_1-y_1z_2)\,,
\]
where $\deg x_i = (1,0)$ and $\deg y_j = (d,1)$, and consequently $\calR_\d$ is the subalgebra of $\calR$ generated by the elements $x_iy_j$. When $K$ has characteristic zero and $z_1$ and $z_2$ are general forms of degree $d$, the results of Section~\ref{section:hypersurfaces} imply that $\calR_\d$ has rational singularities if and only if $d\le m$, and that it is of F-regular type if and only if $d<m$. As a consequence, we obtain large families of rings of the form $\calR_\d$, standard graded over a field, which have rational singularities, but which are not of F-regular type.

It is worth pointing out that if $\calR$ is an $\NN^2$-graded ring over an infinite field $\calR_{(0,0)}=K$, and $\d=(g,h)\ZZ$ for coprime positive integers $g$ and $h$, then $\calR_\d$ is the ring of invariants of the torus $K^*$ acting on $\calR$ via
\[
\lambda\colon r\mapsto\lambda^{hi-gj}r\qquad\text{ where }\lambda \in K^*\text{ and }r\in \calR_{(i,j)}\,.
\]
Consequently there exist torus actions on hypersurfaces for which the rings of invariants have rational singularities but are not of F-regular type.

In Section~\ref{section:classgroup} we use diagonal subalgebras to construct standard graded normal rings $R$, with isolated singularities, for which ${H^2_\frakm(R)}_0=0$ and ${H^2_\frakm(R)}_1\neq0$. If $S$ is the localization of such a ring $R$ at its homogeneous maximal ideal, then, by Danilov's results, the divisor class group of $S$ is a finitely generated abelian group, though $S$ does not have a discrete divisor class group. Such rings $R$ are also of interest in view of the results of \cite{RSS}, where it is proved that the image of~${H^2_\frakm(R)}_0$ in $H^2_\frakm(R^+)$ is annihilated by elements of $R^+$ of arbitrarily small positive degree; here $R^+$ denotes the absolute integral closure of $R$. A corresponding result for ${H^2_\frakm(R)}_1$ is not known at this point, and the rings constructed in Section~\ref{section:classgroup} constitute interesting test cases.

Section~\ref{section:prelim} summarizes some notation and conventions for multigraded rings and modules. In Section~\ref{section:hypersurfaces} we carry out an analysis of diagonal subalgebras of bigraded hypersurfaces; this uses results on rational singularities and F-regular rings proved in Sections~\ref{section:rational} and~\ref{section:Freg} respectively.

The authors would like to thank Shiro Goto and Ken-ichi Yoshida
for their valuable comments.

\section{Preliminaries}
\label{section:prelim}

In this section, we provide a brief treatment of multigraded rings and modules; see \cite{GW1,GW2,HHR}, and \cite{HIO} for further details.

By an $\NN^r$-\emph{graded ring} we mean a ring
\[
\calR=\bigoplus_{\bsn\in\NN^r}\calR_\bsn\,,
\]
which is finitely generated over the subring $\calR_\bszero$. If $(\calR_\bszero,\frakm)$ is a local ring, then $\calR$ has a unique homogeneous maximal ideal $\frakM=\frakm \calR+\calR_+$, where $\calR_+=\oplus_{\bsn\neq\bszero}\calR_\bsn$.

For $\bsm=(m_1,\dots,m_r)$ and $\bsn=(n_1,\dots,n_r)$ in $\ZZ^r$, we say $\bsn>\bsm$ (resp. $\bsn\ge\bsm$) if $n_i>m_i$ (resp. $n_i\ge m_i$) for each $i$.

Let $M$ be a $\ZZ^r$-graded $\calR$-module. For $\bsm\in\ZZ^r$, we set
\[
M_{\ge\bsm}=\bigoplus_{\bsn\ge\bsm}M_\bsn\,,
\]
which is a $\ZZ^r$-graded submodule of $M$. One writes $M(\bsm)$ for the $\ZZ^r$-graded $\calR$-module with shifted grading ${[M(\bsm)]}_\bsn=M_{\bsm+\bsn}$ for each $\bsn\in\ZZ^r$.

Let $M$ and $N$ be $\ZZ^r$-graded $\calR$-modules. Then $\HH_\calR(M,N)$ is the $\ZZ^r$-graded module with ${[\HH_\calR(M,N)]}_\bsn$ being the abelian group consisting of degree preserving $\calR$-linear homomorphisms from $M$ to $N(\bsn)$. 

The functor $\HE_\calR^i(M,-)$ is the $i$-th derived functor of $\HH_\calR(M,-)$ in the category of $\ZZ^r$-graded $\calR$-modules. When $M$ is finitely generated, $\HE_\calR^i(M,N)$ and $\Ext_\calR^i(M,N)$ agree as underlying $\calR$-modules. For a homogeneous ideal $\fraka$ of $\calR$, the local cohomology modules of $M$ with support in $\fraka$ are the $\ZZ^r$-graded modules
\[
H^i_\fraka(M)=\varinjlim_n\,\HE^i_\calR(\calR/\fraka^n,M)\,.
\]

Let $\phi\colon\ZZ^r\to\ZZ^s$ be a homomorphism of abelian groups satisfying $\phi(\NN^r)\subseteq\NN^s$. We write $\calR^\phi$ for the ring $\calR$ with the $\NN^s$-grading where
\[
{[\calR^\phi]}_\bsn=\bigoplus_{\phi(\bsm)=\bsn}\calR_\bsm\,.
\]
If $M$ is a $\ZZ^r$-graded $\calR$-module, then $M^\phi$ is the $\ZZ^s$-graded $\calR^\phi$-module with 
\[
{[M^\phi]}_\bsn=\bigoplus_{\phi(\bsm)=\bsn}M_\bsm\,.
\]
The change of grading functor $(-)^\phi$ is exact; by \cite[Lemma~1.1]{HHR} one has
\[
H^i_\frakM(M)^\phi=H^i_{\frakM^\phi}(M^\phi)\,.
\]
Consider the projections $\phi_i\colon\ZZ^r\to\ZZ$ with $\phi_i(m_1,\dots,m_r)=m_i$, and set
\[
a(\calR^{\phi_i})=\max\left\{a\in\ZZ\mid\left[H^{\dim\calR}_\frakM(\calR)^{\phi_i}\right]_a\neq0\right\};
\]
this is the $a$-invariant of the $\NN$-graded ring $\calR^{\phi_i}$ in the sense of Goto and Watanabe \cite{GW1}. As in \cite{HHR}, the \emph{multigraded $\bsa$-invariant} of $\calR$ is 
\[
\bsa(\calR)=\big(a(\calR^{\phi_1}),\dots,a(\calR^{\phi_r})\big)\,.
\]

Let $\calR$ be a $\ZZ^2$-graded ring and let $g,h$ be positive integers. The subgroup $\d=(g,h)\ZZ$ is a \emph{diagonal} in $\ZZ^2$, and the corresponding \emph{diagonal subalgebra} of $\calR$ is
\[
\calR_\d=\bigoplus_{k\in\ZZ}\calR_{(gk,hk)}\,.
\]
Similarly, if $M$ is a $\ZZ^2$-graded $\calR$-module, we set
\[
M_\d=\bigoplus_{k\in\ZZ}M_{(gk,hk)}\,,
\]
which is a $\ZZ$-graded module over the $\ZZ$-graded ring $\calR_\d$.

\begin{lemma}
\label{lemma:segre}
Let $A$ and $B$ be $\NN$-graded normal rings, finitely generated over a field $A_0=K=B_0$. Set $T=A\otimes_KB$. Let $g$ and $h$ be positive integers and set $\d=(g,h)\ZZ$. Let $\fraka$, $\frakb$, and $\frakm$ denote the homogeneous maximal ideals of $A$, $B$, and $T_\d$ respectively. Then, for each $q\ge0$ and $i,j,k\in\ZZ$, one has
\begin{multline*}
{H^q_\frakm\big({T(i,j)_\d}\big)}_k=\big(A_{i+gk}\otimes{H^q_\frakb(B)}_{j+hk}\big)\oplus\big({H^q_\fraka(A)}_{i+gk}\otimes B_{j+hk}\big)\,\oplus\\
\bigoplus_{q_1+q_2=q+1}\big({H^{q_1}_\fraka(A)}_{i+gk}\otimes{H^{q_2}_\frakb(B)}_{j+hk}\big)\,.
\end{multline*}
\end{lemma}

\begin{proof}
Let $A^{(g)}$ and $B^{(h)}$ denote the respective Veronese subrings of $A$ and $B$. Set
\[
A^{(g,i)}=\bigoplus_{k\in\ZZ} A_{i+gk}\qquad\text{and}\qquad B^{(h,j)}=\bigoplus_{k\in\ZZ} B_{j+hk}\,,
\]
which are graded $A^{(g)}$ and $B^{(h)}$ modules respectively. Using $\seg$ for the Segre product,
\[
{T(i,j)}_\d=\bigoplus_{k\in\ZZ}A_{i+gk}\otimes_KB_{j+hk}=A^{(g,i)}\,\seg\,B^{(h,j)}\,.
\]
The ideal $A^{(g)}_+A$ is $\fraka$-primary; likewise, $B^{(h)}_+B$ is $\frakb$-primary. The K\"unneth formula for local cohomology, \cite[Theorem~4.1.5]{GW1}, now gives the desired result.
\end{proof}

\begin{notation}
We use bold letters to denote lists of elements, e.g., $\bsz=z_1,\dots,z_s$ and $\bsgamma=\gamma_1,\dots,\gamma_s$.
\end{notation}

\section{Diagonal subalgebras of bigraded hypersurfaces}
\label{section:hypersurfaces}

We prove the following theorem about diagonal subalgebras of $\NN^2$-graded hypersurfaces. The proof uses results proved later in Sections~\ref{section:rational} and~\ref{section:Freg}.

\begin{theorem}
\label{theorem:hypersurface}
Let $K$ be a field, let $m,n$ be integers with $m,n\ge 2$, and let
\[
\calR=K[x_1,\dots,x_m,y_1,\dots,y_n]/(f)
\]
be a normal $\NN^2$-graded hypersurface where $\deg x_i=(1,0)$, $\deg y_j=(0,1)$, and $\deg f=(d,e)>(0,0)$. For positive integers $g$ and $h$, set $\d=(g,h)\ZZ$. Then:
\begin{enumerate}
\item The ring $\calR_\d$ is Cohen-Macaulay if and only if $\lfloor(d-m)/g\rfloor<e/h$ and $\lfloor(e-n)/h\rfloor<d/g$. In particular, if $d<m$ and $e<n$, then $\calR_\d$ is Cohen-Macaulay for each diagonal $\d$.

\item The graded canonical module of $\calR_\d$ is ${\calR(d-m,e-n)}_\d$. Hence $\calR_\d$ is Gorenstein if and only if $(d-m)/g=(e-n)/h$, and this is an integer.
\end{enumerate}
If $K$ has characteristic zero, and $f$ is a generic polynomial of degree $(d,e)$, then:
\begin{enumerate}
\setcounter{enumi}{2}
\item The ring $\calR_\d$ has rational singularities if and only if it is Cohen-Macaulay and $d<m$ or $e<n$.

\item The ring $\calR_\d$ is of F-regular type if and only if $d<m$ and $e<n$.
\end{enumerate}
\end{theorem}

For $m,n\ge 3$ and $\d=(1,1)\ZZ$, the properties of $\calR_\d$, as determined by $m,n,d,e$, are summarized in Figure~1.
\begin{figure}[ht]
\label{fig}
\psfrag{d}{$d$}
\psfrag{e}{$e$}
\psfrag{m-1}{$m-1$}
\psfrag{n-1}{$n-1$}
\psfrag{e=d-m+1}{$e=d-m+1$}
\psfrag{e=d+n-1}{$e=d+n-1$}
\psfrag{Cohen-Macaulay}{Cohen-Macaulay}
\psfrag{Rational Singularities}{Rational Singularities}
\psfrag{F-regular type}{F-regular type}
\psfrag{Gorenstein: e=d-m+n}{Gorenstein: $e=d-m+n$}
\includegraphics[scale=.4]{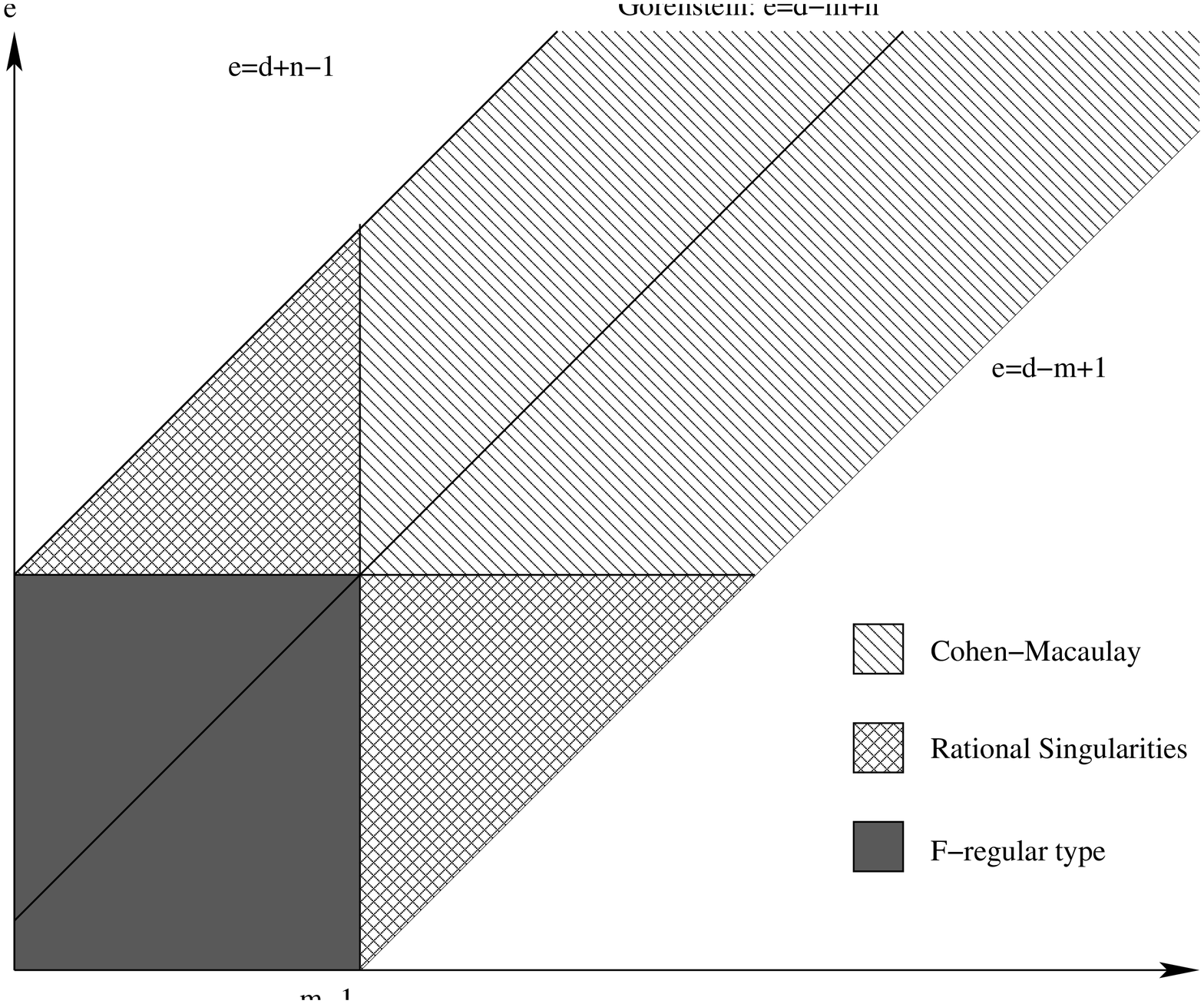}
\caption{Properties of $\calR_\d$ for $\d=(1,1)\ZZ$.}
\end{figure}

\begin{remark}
Let $m,n\ge2$. A generic hypersurface of degree $(d,e)>(0,0)$ in $m,n$ variables is normal precisely when
\[
m>\min(2,d)\qquad\text{ and }\qquad n>\min(2,e)\,.
\]
Suppose that $m=2=n$, and that $f$ is nonzero. Then $\dim\calR_\d=2$; since $\calR_\d$ is generated over a field by elements of equal degree, $\calR_\d$ is of F-regular type if and only if it has rational singularities; see \cite{KW:dim2}. This is the case precisely if
\begin{center}
\begin{tabular}{lll}
$d=1,\ e\le h+1$\,, & or \\
$e=1,\ d\le g+1$\,.\\
\end{tabular}
\end{center}

Following a suggestion of Hara, the case $n=2$ and $e=1$ was used in \cite[Example~7.3]{Singh:PJM} to construct examples of standard graded rings with rational singularities which are not of F-regular type. 
\end{remark}

\begin{proof}[Proof of Theorem~\ref{theorem:hypersurface}]
Set $A=K[\bsx]$, $B=K[\bsy]$, and $T=A\otimes_KB$. By Lemma~\ref{lemma:segre}, $H^q_\frakm(T_\d)=0$ for $q\neq m+n-1$. The local cohomology exact sequence induced by
\[\CD
0@>>>{T(-d,-e)}_\d@>f>>T_\d@>>>\calR_\d@>>>0
\endCD\]
therefore gives $H^{q-1}_\frakm(\calR_\d)=H^q_\frakm({T(-d,-e)}_\d)$ for $q\le m+n-2$, and also shows that $H^{m+n-2}_\frakm(\calR_\d)$ and $H^{m+n-1}_\frakm(\calR_\d)$ are, respectively, the kernel and cokernel of
\[\CD
H^{m+n-1}_\frakm({T(-d,-e)}_\d)@>f>>H^{m+n-1}_\frakm(T_\d)\\
@| @|\\
{[H^m_\fraka(A(-d))\otimes H^n_\frakb(B(-e))]}_\d@>f>>{[H^m_\fraka(A)\otimes H^n_\frakb(B)]}_\d\,.
\endCD\]
The horizontal map above is surjective since its graded dual
\[\CD
{[A(d-m)\otimes B(e-n)]}_\d @<f<<{[A(-m)\otimes B(-n)]}_\d\\
@| @|\\
{T(d-m,e-n)}_\d @<f<<{T(-m,-n)}_\d
\endCD\]
is injective. In particular, $\dim\calR_\d=m+n-2$.

It follows from the above discussion that $\calR_\d$ is Cohen-Macaulay if and only if $H^q_\frakm({T(-d,-e)}_\d)=0$ for each $q\le m+n-2$. By Lemma~\ref{lemma:segre}, this is the case if and only if, for each integer $k$, one has
\[
A_{-d+gk}\otimes{H^n_\frakb(B)}_{-e+hk}=0={H^m_\fraka(A)}_{-d+gk}\otimes B_{-e+hk}\,.
\]
Hence $\calR_\d$ is Cohen-Macaulay if and only if there is no integer $k$ satisfying
\[
d/g\le k\le(e-n)/h\qquad\text{ or }\qquad e/h\le k\le(d-m)/g\,,
\]
which completes the proof of (1).

For (2), note that the graded canonical module of $\calR_\d$ is the graded dual of $H^{m+n-2}_\frakm(\calR_\d)$, and hence that it equals
\[
\coker\big({T(-m,-n)}_\d\overset{f}\to{T(d-m,e-n)}_\d\big)={\calR(d-m,e-n)}_\d\,.
\]
This module is principal if and only if ${\calR(d-m,e-n)}_\d=\calR_\d(a)$ for some integer~$a$, i.e., $d-m=ga$ and $e-n=ha$.

When $f$ is a general polynomial of degree $(d,e)$, the ring $\calR_\d$ has an isolated singularity. Also, $\calR_\d$ is normal since it is a direct summand of the normal ring $\calR$. By Theorem~\ref{thm:FlennerWatanabe}, $\calR_\d$ has rational singularities precisely if it is Cohen-Macaulay and $a(\calR_\d)<0$; this proves (3).

It remains to prove (4). If $d<m$ and $e<n$, then Theorem~\ref{thm:multigraded} implies that $\calR$ has rational singularities. By Theorem~\ref{thm:HSWMS}, it follows that for almost all primes $p$, the characteristic $p$ models $\calR_p$ of $\calR$ are F-rational hypersurfaces which, therefore, are F-regular. Alternatively, $\calR_p$ is a generic hypersurface of degree $(d,e)<(m,n)$, so Theorem~\ref{thm:general-bigraded} implies that $\calR_p$ is F-regular. Since ${(\calR_p)}_\d$ is a direct summand of $\calR_p$, it follows that ${(\calR_p)}_\d$ is F-regular. The rings ${(\calR_p)}_\d$ are characteristic $p$ models of $\calR_\d$, so we conclude that $\calR_\d$ is of F-regular type.

Suppose $\calR_\d$ has F-regular type, and let ${(\calR_p)}_\d$ be a characteristic $p$ model which is F-regular. Fix an integer $k>d/g$. Then Proposition~\ref{proposition:freg} implies that there exists an integer $q=p^e$ such that 
\[
\rank_K{({(\calR_p)}_\d)}_k\le\rank_K{[H^{m+n-2}_\frakm(\omega^{(q)})]}_k\,,
\]
where $\omega$ is the graded canonical module of ${(\calR_p)}_\d$. Using (2), we see that
\[
H^{m+n-2}_\frakm(\omega^{(q)})=H^{m+n-2}_\frakm({\calR_p(qd-qm,qe-qn)}_\d)\,.
\]
Let $T_p$ be a characteristic $p$ model for $T$ such that $T_p/fT_p=\calR_p$. Multiplication by $f$ on $T_p$ induces a local cohomology exact sequence
\begin{multline*}
\cdots\to H^{m+n-2}_{\frakm_p}{(T_p(qd-qm,qe-qn)}_\d)\to H^{m+n-2}_{\frakm_p}({\calR_p(qd-qm,qe-qn)}_\d)\\
\to H^{m+n-1}_{\frakm_p}{(T_p(qd-qm-d,qe-qn-e)}_\d)\to\cdots\,.
\end{multline*}
Since $H^{m+n-2}_{\frakm_p}{(T_p(qd-qm,qe-qn)}_\d)$ vanishes by Lemma~\ref{lemma:segre}, we conclude that
\begin{align*}
\rank_K{({(\calR_p)}_\d)}_k&\le\rank_K{[H^{m+n-1}_{\frakm_p}{(T_p(qd-qm-d,qe-qn-e)}_\d)]}_k\\
&=\rank_K{H^m_{\fraka_p}(A_p)}_{qd-qm-d+gk}\otimes{H^n_{\frakb_q}(B_p)}_{qe-qn-e+hk}\,.
\end{align*}
Hence $qd-qm-d+gk<0$; as $d-gk<0$, we conclude $d<m$. Similarly, $e<n$.
\end{proof}

We conclude this section with an example where a local cohomology module of a standard graded ring is not rigid in the sense that ${H^2_\frakm(R)}_0=0$ while ${H^2_\frakm(R)}_1\neq0$. Further such examples are constructed in Section~\ref{section:classgroup}.

\begin{proposition}
Let $K$ be a field and let
\[
\calR=K[x_1,x_2,x_3,y_1,y_2]/(f)
\]
where $\deg x_i=(1,0)$, $\deg y_j=(0,1)$, and $\deg f=(d,e)$ for $d\ge4$ and $e\ge 1$. Let $g$ and $h$ be positive integers such that $g\le d-3$ and $h\ge e$, and set $\d=(g,h)\ZZ$. Then ${H^2_\frakm(\calR_\d)}_0=0$ and ${H^2_\frakm(\calR_\d)}_1\neq0$.
\end{proposition}

\begin{proof}
Using the resolution of $\calR$ over the polynomial ring $T$ as in the proof of Theorem~\ref{theorem:hypersurface}, we have an exact sequence
\[
H^2_\frakm(T_\d)\to H^2_\frakm(\calR_\d)\to H^3_\frakm(T(-d,-e)_\d)\to H^3_\frakm(T_\d)\,.
\]
Lemma~\ref{lemma:segre} implies that $H^2_\frakm(T_\d)=0=H^3_\frakm(T_\d)$. Hence, again by Lemma~\ref{lemma:segre},
\[
{H^2_\frakm(\calR_\d)}_0={H^3(A)}_{-d}\otimes B_{-e}=0\quad\text{ and }\quad
{H^2_\frakm(\calR_\d)}_1={H^3(A)}_{g-d}\otimes B_{h-e}\neq 0\,.\qedhere
\]
\end{proof}

\section{Non-rigid local cohomology modules}
\label{section:classgroup}

We construct examples of standard graded normal rings $R$ over~$\CC$, with only isolated singularities, for which ${H^2_\frakm(R)}_0=0$ and ${H^2_\frakm(R)}_1 \neq 0$. Let $S$ be the localization of such a ring $R$ at its homogeneous maximal ideal. By results of Danilov \cite{Danilov1,Danilov2}, Theorem~\ref{thm:Danilov} below, it follows that the divisor class group of $S$ is finitely generated, though $S$ does not have a discrete divisor class group, i.e., the natural map $\Cl(S)\to\Cl(S[[t]])$ is not bijective. Here, remember that if $A$ is a Noetherian normal domain, then so is $A[[t]]$.

\begin{theorem}
\label{thm:Danilov}
Let $R$ be a standard graded normal ring, which is finitely generated as an algebra over $R_0=\CC$. Assume, moreover, that $X=\Proj R$ is smooth. Set $(S,\frakm)$ to be the local ring of $R$ at its homogeneous maximal ideal, and $\widehat{S}$ to be the $\frakm$-adic completion of $S$. Then
\begin{enumerate}
\item the group $\Cl(S)$ is finitely generated if and only if $H^1(X,\calO_X)=0$;
\item the map $\Cl(S)\to\Cl(\hat{S})$ is bijective if and only if $H^1(X,\calO_X(i))=0$ for each integer $i\ge 1$; and
\item the map $\Cl(S)\to\Cl(S[[t]])$ is bijective if and only if $H^1(X,\calO_X(i))=0$ for each integer $i\ge 0$.
\end{enumerate}
\end{theorem}

The essential point in our construction is in the following proposition:

\begin{theorem}
\label{thm:rigidity}
Let $A$ be a Cohen-Macaulay ring of dimension $d\ge 2$, which is a standard graded algebra over a field $K$. For $s\ge 2$, let $z_1,\dots,z_s$ be a regular sequence in $A$, consisting of homogeneous elements of equal degree, say $k$. Consider the Rees ring $\calR=A[z_1t,\dots,z_st]$ with the $\ZZ^2$-grading where $\deg x=(n,0)$ for $x\in A_n$, and $\deg z_it=(0,1)$.

Let $\d=(g,h)\ZZ$ where $g,h$ are positive integers, and let $\frakm$ denote the homogeneous maximal ideal of $\calR_\d$. Then:
\begin{enumerate}
\item $H^q_\frakm(\calR_\d)=0$ if $q\neq d-s+1, d$; and
\item ${H^{d-s+1}_\frakm(\calR_\d)}_i\neq 0$ if and only if $1\le i\le(a+ks-k)/g$, where $a$ is the $a$-invariant of $A$.
\end{enumerate}
In particular, $\calR_\d$ is Cohen-Macaulay if and only if $g>a+ks-k$.
\end{theorem}

\begin{example}
\label{ex:rigidity}
For $d\ge3$, let $A=\CC[x_0,\dots,x_d]/(f)$ be a standard graded hypersurface such that $\Proj A$ is smooth over $\CC$. Take general $k$-forms $z_1,\dots,z_{d-1}\in A$, and consider the Rees ring $\calR=A[z_1t,\dots,z_{d-1}t]$. Since $(\bsz)\subset A$ is a radical ideal, 
\[
\gr((\bsz),A)\cong A/(\bsz)[y_1,\dots,y_{d-1}]
\]
is a reduced ring, and therefore $\calR=A[z_1t,\dots,z_{d-1}t]$ is integrally closed in $A[t]$. Since $A$ is normal, so is $\calR$. Note that $\Proj\calR_\d$ is the blow-up of $\Proj A$ at the subvariety defined by $(\bsz)$, i.e., at $k^{d-1}(\deg f)$ points. It follows that $\Proj \calR_\d$ is smooth over $\CC$. Hence $\calR_\d$ is a standard graded $\CC$-algebra, which is normal and has an isolated singularity. 

If $\d=(g,h)\ZZ$ is a diagonal with $1\le g\le \deg f+k(d-2)-(d+1)$ and $h\ge1$, then Theorem~\ref{thm:rigidity} implies that
\[
{H_\frakm^2(\calR_\d)}_0=0\quad\text{ and }\quad{H_\frakm^2(\calR_\d)}_1\neq 0\,.
\]
\end{example}

The rest of this section is devoted to proving Theorem~\ref{thm:rigidity}. We may assume that the base field $K$ is infinite. Then one can find linear forms $x_1,\dots,x_{d-s}$ in $A$ such that $x_1,\dots,x_{d-s},z_1,\dots,z_s$ is a maximal $A$-regular sequence.

We will use the following lemma; the notation is as in Theorem~\ref{thm:rigidity}.

\begin{lemma}
\label{lem:rigidity}
Let $\fraka$ be the homogeneous maximal ideal of $A$. Set $I=(z_1,\dots,z_s)A$. Let $r$ be a positive integer.
\begin{enumerate}
\item $H^q_\fraka(I^r) = 0$ if $q \neq d-s+1, d$.
\item Assume $d > s$. Then, $H^{d-s+1}_\fraka(I^r)_i\neq 0$ if and only if $i\le a+ks+rk-k$.
\item Assume $d = s$. Then, $H^{d-s+1}_\fraka(I^r)_i\neq 0$ if and only if $0\le i\le a+ks+rk-k$.
\end{enumerate}
\end{lemma}

\begin{proof}
Recall that $A$ and $A/I^r$ are Cohen-Macaulay rings of dimension $d$ and $d-s$, respectively. By the exact sequence
\[
0\to I^r\to A\to A/I^r\to 0
\]
we obtain
\[
H^q_\fraka(I^r)=\begin{cases}
H^d_\fraka(A)&\text{ if } q=d\\
H^{d-s}_\fraka(A/I^r)&\text{ if } q=d-s+1\\
0&\text{ if }q\neq d-s+1,d\,,
\end{cases}
\]
which proves (1).

Next we prove (2) and (3). Since $A/I^r$ is a standard graded Cohen-Macaulay ring of dimension $d-s$, it is enough to show that the $a$-invariant of this ring equals $a+ks+rk-k$. This is straightforward if $r=1$, and we proceed by induction. Consider the exact sequence
\[
0\to I^r/I^{r+1}\to A/I^{r+1}\to A/I^r\to 0\,.
\]
Since $z_1,\dots,z_s$ is a regular sequence of $k$-forms, $I^r/I^{r+1}$ is isomorphic to 
\[
\big((A/I)(-rk)\big)^{\binom{s-1+r}{r}}\,.
\]
Thus, we have the following exact sequence:
\[
0\to H^{d-s}_\fraka\big((A/I)(-rk)\big)^{\binom{s-1+r}{r}}\to H^{d-s}_\fraka(A/I^{r+1})\to H^{d-s}_\fraka(A/I^r)\to 0\,.
\]
The $a$-invariant of $(A/I)(-rk)$ equals $a+ks+rk$, and that of $A/I^r$ is $a+ks+rk-k$ by the inductive hypothesis. Thus, $A/I^{r+1}$ has $a$-invariant $a+ks+rk$.
\end{proof}

\begin{proof}[Proof of Theorem~\ref{thm:rigidity}]
Let $B=K[y_1,\dots,y_s]$ be a polynomial ring, and set
\[
T=A\otimes_K B=A[y_1,\dots,y_s]\,.
\]
Consider the $\ZZ^2$-grading on $T$ where $\deg x=(n,0)$ for $x\in A_n$, and $\deg y_i=(0,1)$ for each $i$. One has a surjective homomorphism of graded rings
\[
T\to\calR=A[z_1t,\dots,z_st]\qquad\text{ where }y_i\mapsto z_it\,,
\]
and this induces an isomorphism
\[
\calR\cong T/I_2\!\left(\begin{smallmatrix}z_1&\dots&z_s\cr y_1&\dots&y_s\end{smallmatrix}\right)\,.
\]
The minimal free resolution of $\calR$ over $T$ is given by the Eagon-Northcott complex
\[
0\to F^{-(s-1)}\to F^{-(s-2)}\to\cdots\to F^{0}\to 0\,,
\]
where $F^0=T(0,0)$, and $F^{-i}$ for $1\le i\le s-1$ is the direct sum of $\binom{s}{i+1}$ copies of
\[
T(-k,-i)\oplus T(-2k,-(i-1))\oplus\cdots\oplus T(-ik,-1)\,.
\]
Let $\frakn$ be the homogeneous maximal ideal of $T_\d$. One has the spectral sequence:
\[
E^{p,q}_2=H^p(H^q_\frakn(F^\bullet_\d))\Longrightarrow H^{p+q}_\frakm(\calR_\d)\,.
\]

Let $G$ be the set of $(n,m)$ such that $T(n,m)$ appears in the Eagon-Northcott complex above, i.e., the elements of $G$ are
\[
\begin{array}{c}
(0,0)\,,\\
(-k,-1),\ \\
(-k,-2),\ (-2k,-1)\,,\\
(-k,-3),\ (-2k,-2),\ (-3k,-1)\,,\\
\vdots\\
(-k,-(s-1)),\qquad\dots\quad\ (-(s-1)k,-1)\,.
\end{array}
\]

Let $\fraka$ and $\frakb$ be the homogeneous maximal ideal of $A$ and $B$ respectively. For integers $n$ and $m$, the K\"unneth formula gives
\begin{align*}
H&^q_\frakn(T(n,m))\\
&=H^q_\frakn(A(n)\otimes_K B(m))\\
&=\big(H^q_\fraka(A(n))\otimes B(m)\big)\oplus\big(A(n)\otimes H^q_\frakb(B(m))\big)\oplus\bigoplus_{i+j=q+1}H^i_\fraka(A(n))\otimes H^j_\frakb(B(m))\\
&=H^q_\fraka(T(n,m))\oplus H^q_\frakb(T(n,m))\oplus
\bigoplus_{i+j=q+1}H^i_\fraka(A(n))\otimes_KH^j_\frakb(B(m))\,.
\end{align*}
As $A$ and $B$ are Cohen-Macaulay of dimension $d$ and $s$ respectively, it follows that
\[
H^q_\frakn(F^\bullet)=0\qquad\text{ if } q\neq s,d,d+s-1\,.
\]
In the case where $d>s$, one has
\[
H^s_\frakn(F^\bullet)=H^s_\frakb(F^\bullet)\quad\text{ and }\quad H^d_\frakn(F^\bullet)=H^d_\fraka(F^\bullet)\,,
\]
and if $d=s$, then
\[
H^d_\frakn(F^\bullet)=H^d_\fraka(F^\bullet)\oplus H^s_\frakb(F^\bullet)\,.
\]

We claim ${H^s_\frakb(F^\bullet)}_\d=0$. If not, there exists $(n,m)\in G$ and $\ell\in\ZZ$ such that
\[
{H^s_\frakb(T(n,m))}_{(g\ell,h\ell)}\neq 0\,.
\]
This implies that
\[
{H^s_\frakb(T(n,m))}_{(g\ell,h\ell)}={A(n)}_{g\ell}\otimes_K{H^s_\frakb(B(m))}_{h\ell}=A_{n+g\ell}\otimes_K{H^s_\frakb(B)}_{m+h\ell}
\]
is nonzero, so
\[
n+g\ell\ge 0\quad\text{ and }\quad m+h\ell\le-s\,,
\]
and hence
\[
-\frac{n}{g}\le\ell\le-\frac{s+m}{h}\,.
\]
But $(n,m)\in G$, so $n\le 0$ and $m\ge-(s-1)$, implying that
\[
0\le\ell\le-\frac{1}{h}\,,
\]
which is not possible. This proves that ${H^s_\frakb(F^\bullet)}_\d=0$. Thus, we have
\[
{H^q_\frakn(F^\bullet)}_\d=\begin{cases}
0&\text{ if }q\neq d, d+s-1,\\
{H^d_\fraka(F^\bullet)}_\d&\text{ if }q=d.
\end{cases}
\]
It follows that
\[
E_2^{p,q}=H^p(H^q_\frakn(F^\bullet_\d))=E_\infty^{p,q}
\]
for each $p$ and $q$. Therefore,
\[
H^i_\frakm(\calR_\d)=E_2^{i-d,d}=H^{i-d}(H^d_\frakn(F^\bullet_\d))=H^{i-d}({H^d_\fraka(F^\bullet)}_\d)={H^i_\fraka(\calR)}_\d
\]
for $d-s+1\le i\le d-1$, and
\[
H^i_\frakm(\calR_\d)=0\qquad\text{ for }i<d-s+1\,.
\]

We next study $H^i_\fraka(\calR)$. Since
\[
\calR=A\oplus I(k)\oplus I^2(2k)\oplus\cdots\oplus I^r(rk)\oplus\cdots\,,
\]
we have
\[
H^i_\fraka(\calR)=H^i_\fraka(A)\oplus H^i_\fraka(I)(k)\oplus H^i_\fraka(I^2)(2k)\oplus\cdots\oplus H^i_\fraka(I^r)(rk)\oplus\cdots\,.
\]
Theorem~\ref{thm:rigidity}~(1) now follow using Lemma~\ref{lem:rigidity}~(1).

Assume that $d>s$. Then, by Lemma~\ref{lem:rigidity}~(2), $H^{d-s+1}_\fraka(I^r(rk))_i \neq 0$ if and only if $i \le a+ks-k$.

Assume that $d=s$. Then, by Lemma~\ref{lem:rigidity}~(3), $H^{d-s+1}_\fraka(I^r(rk))_i \neq 0$ if and only if $-rk \le i \le a+ks-k$.

In each case, $H^{d-s+1}_\fraka(\calR)_{(gi,hi)} \neq 0$ if and only if 
\[
1\le i\le \frac{a+ks-k}{g}\,.\qedhere
\]
\end{proof}

\section{Rational singularities}
\label{section:rational}

Let $R$ be a normal domain, essentially of finite type over a field of characteristic zero, and consider a \emph{desingularization} $f\colon Z\to\Spec R$, i.e., a proper birational morphism with $Z$ a nonsingular variety. One says $R$ has \emph{rational singularities} if $R^if_*\calO_Z=0$ for each $i\ge 1$; this does not depend on the choice of the desingularization $f$. For $\NN$-graded rings, one has the following criterion due to Flenner \cite{Flenner} and Watanabe \cite{KW:kstar}.

\begin{theorem}
\label{thm:FlennerWatanabe}
Let $R$ be a normal $\NN$-graded ring which is finitely generated over a field $R_0$ of characteristic zero. Then $R$ has rational singularities if and only if it is Cohen-Macaulay, $a(R)<0$, and the localization $R_\frakp$ has rational singularities for each $\frakp\in\Spec R\smallsetminus\{R_+\}$.
\end{theorem}

When $R$ has an isolated singularity, the above theorem gives an effective criterion for determining if $R$ has rational singularities. However, a multigraded hypersurface typically does not have an isolated singularity, and the following variation turns out to be useful:

\begin{theorem}
\label{thm:multigraded}
Let $R$ be a normal $\NN^r$-graded ring such that $R_\bszero$ is a local ring essentially of finite type over a field of characteristic zero, and $R$ is generated over $R_\bszero$ by elements 
\[
x_{11},x_{12},\dots,x_{1t_1},\quad x_{21},x_{22},\dots,x_{2t_2},\quad\dots,\quad x_{r1},x_{r2},\dots,x_{rt_r}\,,
\]
where $\deg x_{ij}$ is a positive integer multiple of the $i$-th unit vector $e_i\in\NN^r$. Then $R$ has rational singularities if and only if
\begin{enumerate}
\item $R$ is Cohen-Macaulay,
\item $R_\frakp$ has rational singularities for each $\frakp$ belonging to
\[
\Spec R\setminus\big(V(x_{11},x_{12},\dots,x_{1t_1})\cup\dots\cup V(x_{r1},x_{r2},\dots,x_{rt_r})\big), \quad\text{and}
\]
\item $\bsa(R)<\bszero$, i.e., $a(R^{\phi_i})<0$ for each coordinate projection $\phi_i\colon \NN^r\to\NN$.
\end{enumerate}
\end{theorem}

Before proceeding with the proof, we record some preliminary results.

\begin{remark}
\label{rem:natural}
Let $R$ be an $\NN$-graded ring. We use $R^\nat$ to denote the Rees algebra with respect to the filtration $F_n=R_{\ge n}$, i.e., 
\[
R^\nat=F_0\oplus F_1T\oplus F_2T^2\oplus\cdots\,.
\]
When considering $\Proj R^\nat$, we use the $\NN$-grading on $R^\nat$ where ${[R^\nat]}_n=F_nT^n$. The inclusion 
$R={[R^\nat]}_0\hookrightarrow R^\nat$ gives a map
\[
\Proj R^\nat\overset{f}\to\Spec R\,.
\]
Also, the inclusions $R_n\hookrightarrow F_n$ give rise to an injective homomorphism of graded rings $R \hookrightarrow R^\nat$, which induces a surjection
\[
\Proj R^\nat\overset{\pi}\to\Proj R\,.
\]
\end{remark}

\begin{lemma}
\label{lem:natural}
Let $R$ be an $\NN$-graded ring which is finitely generated over $R_0$, and assume that $R_0$ is essentially of finite type over a field of characteristic zero.

If $R_\frakp$ has rational singularities for all primes $\frakp\in\Spec R\smallsetminus V(R_+)$, then $\Proj R^\nat$ has rational singularities.
\end{lemma}

\begin{proof}
Note that $\Proj R^\nat$ is covered by affine open sets $D_+(r T^n)$ for integers $n\ge 1$ and homogeneous elements $r\in R_{\ge n}$. Consequently, it suffices to check that ${[R^\nat_{r T^n}]}_0$ has rational singularities. Next, note that
\[
\big[R^\nat_{r T^n}\big]_0=R+\frac{1}{r}{[R]}_{\ge n}+\frac{1}{r^2}{[R]}_{\ge 2n}+\cdots.
\]
In the case $\deg r>n$, the ring above is simply $R_r$, which has rational singularities by the hypothesis of the lemma. If $\deg r=n$, then 
\[
{\big[R^\nat_{r T^n}\big]}_0={\big[R_r\big]}_{\ge 0}\,.
\]
The $\ZZ$-graded ring $R_r$ has rational singularities and so, by \cite[Lemma 2.5]{KW:kstar}, the ring ${[R_r]}_{\ge 0}$ has rational singularities as well.
\end{proof}

\begin{lemma}\cite[Lemma~2.3]{Hyry:TAMS}
\label{lem:hyry}
Let $R$ be an $\NN$-graded ring which is finitely generated over a local ring $(R_0,\frakm)$. Suppose ${[H^i_{\frakm+R_+}(R)]}_{\ge 0}=0$ for all $i\ge 0$. Then, for all ideals $\fraka$ of $R_0$, one has
\[
{\big[H^i_{\fraka+R_+}(R)\big]}_{\ge 0}=0\qquad\text{ for all }i\ge 0\,.
\]
\end{lemma}

We are now in a position to prove the following theorem, which is a variation of \cite[Satz~3.1]{Flenner}, \cite[Theorem~2.2]{KW:kstar}, and \cite[Theorem~1.5]{Hyry:MM}.

\begin{theorem}
\label{thm:ratsing}
Let $R$ be an $\NN$-graded normal ring which is finitely generated over~$R_0$, and assume that $R_0$ is a local ring essentially of finite type over a field of characteristic zero. Then $R$ has rational singularities if and only if
\begin{enumerate}
\item $R$ is Cohen-Macaulay,
\item $R_\frakp$ has rational singularities for all $\frakp\in\Spec R\smallsetminus V(R_+)$, and
\item $a(R)<0$.
\end{enumerate}
\end{theorem}

\begin{proof}
It is straightforward to see that conditions~(1)--(3) hold when $R$ has rational singularities, and we focus on the converse. Consider the morphism
\[
Y=\Proj R^\nat\overset{f}\to\Spec R
\]
as in Remark~\ref{rem:natural}. Let $g\colon Z \to Y$ be a desingularization of $Y$; the composition
\[
Z\overset{g}\to Y\overset{f}\to\Spec R
\]
is then a desingularization of $\Spec R$. Note that $Y=\Proj R^\nat$ has rational singularities by Lemma~\ref{lem:natural}, so
\[
g_*{\calO}_Z=\calO_Y\quad\text{ and }\quad R^q g_*\calO_Z=0\qquad\text{ for all } q\ge 1\,.
\]
Consequently the Leray spectral sequence
\[
E_2^{p,q}=H^p(Y,R^q g_*\calO_Z)\Longrightarrow H^{p+q}(Z,\calO_Z)
\]
degenerates, and we get $H^p(Z,\calO_Z)=H^p(Y,\calO_Y)$ for all $p\ge 1$. Since $\Spec R$ is affine, we also have $R^p(g\circ f)_*\calO_Z=H^p(Z,\calO_Z)$. To prove that $R$ has rational singularities, it now suffices to show that $H^p(Y,{\calO}_Y)=0$ for all $p\ge 1$. Consider the map $\pi\colon Y\to X=\Proj R$. We have
\[
H^p(Y,\calO_Y)=H^p(X,\pi_*\calO_X)=\bigoplus_{n\ge 0}H^p(X,\calO_X(n))={\big[H^{p+1}_{R_+}(R)\big]}_{\ge 0}\,.
\]
By condition~(1), we have ${[H^p_{\frakm+R_+}(R)]}_{\ge 0}=0$ for all $p\ge 0$, and so Lemma~\ref{lem:hyry} implies that ${[H^p_{R_+}(R)]}_{\ge 0}=0$ for all $p\ge 0$ as desired.
\end{proof}

\begin{proof}[Proof of theorem~\ref{thm:multigraded}]
If $R$ has rational singularities, it is easily seen that conditions~(1)--(3) must hold. For the converse, we proceed by induction on $r$. The case $r=1$ is Theorem~\ref{thm:ratsing} established above, so assume $r\ge 2$. It suffices to show that $R_\frakM$ has rational singularities where $\frakM$ is the homogeneous maximal ideal of $R$. Set
\[
\frakm=\frakM\cap{\big[R^{\phi_r}\big]}_0\,,
\]
and consider the $\NN$-graded ring $S$ obtained by inverting the multiplicative set ${[R^{\phi_r}]}_0
\smallsetminus\frakm$ in $R^{\phi_r}$. Since $R_\frakM$ is a localization of $S$, it suffices to show that $S$ has rational singularities. Note that $a(S)=a(R^{\phi_r})$, which is a negative integer by (1). Using Theorem~\ref{thm:ratsing}, it is therefore enough to show that $R_\frakP$ has rational singularities for all $\frakP\in\Spec R\smallsetminus V(x_{r1},x_{r2},\dots,x_{rt_r})$. Fix such a prime $\frakP$, and let
\[
\psi\colon\ZZ^r\to\ZZ^{r-1}
\]
be the projection to the first $r-1$ coordinates. Note that $R^\psi$ is the ring $R$ regraded such that $\deg x_{rj}=0$, and the degrees of $x_{ij}$ for $i<r$ are unchanged. Set
\[
\frakp=\frakP\cap{\big[R^\psi\big]}_\bszero\,, 
\]
and let $T$ be the ring obtained by inverting the multiplicative set ${[R^\psi]}_\bszero\smallsetminus\frakp$ in $R^\psi$. It suffices to show that $T$ has rational singularities. Note that $T$ is an $\NN^{r-1}$-graded ring defined over a local ring $(T_\bszero,\frakp)$, and that it has homogeneous maximal ideal $\frakp+\frakb T$ where
\[
\frakb={(R^{\psi})}_+=(x_{ij}\mid i<r)R\,. 
\]
Using the inductive hypothesis, it remains to verify that $\bsa(T)<\bszero$. By condition~(1), for all integers $1\le j\le r-1$, we have
\[
{\big[H^i_\frakM(R)^{\phi_j}\big]}_{\ge 0}=0\qquad\text{ for all }i\ge 0\,,
\]
and using Lemma~\ref{lem:hyry} it follows that 
\[
{\big[H^i_{\frakp+\frakb}(R)^{\phi_j}\big]}_{\ge 0}=0\qquad\text{ for all }i\ge 0\,.
\]
Consequently $a(T^{\phi_j})<0$ for $1\le j\le r-1$, which completes the proof.
\end{proof}

\section{F-regularity}
\label{section:Freg}

For the theory of tight closure, we refer to the papers \cite{HHjams, HHbasec} and \cite{HHjalg}. We summarize results about F-rational and F-regular rings:

\begin{theorem} The following hold for rings of prime characteristic.
\begin{enumerate}
\item Regular rings are F-regular.
\item Direct summands of F-regular rings are F-regular.
\item F-rational rings are normal; an F-rational ring which is a homomorphic image of a Cohen-Macaulay ring is Cohen-Macaulay.
\item F-rational Gorenstein rings are F-regular.
\item Let $R$ be an $\NN$-graded ring which is finitely generated over a field $R_0$. If $R$ is weakly F-regular, then it is F-regular.
\end{enumerate}
\end{theorem}

\begin{proof}
For (1) and (2) see \cite[Theorem~4.6]{HHjams} and \cite[Proposition~4.12]{HHjams} respectively; (3) is part of \cite[Theorem~4.2]{HHbasec}, and for (4) see \cite[Corollary~4.7]{HHbasec}, Lastly, (5) is \cite[Corollary~4.4]{LS}.
\end{proof}

The characteristic zero aspects of tight closure are developed in \cite{HHchar0}. Let $K$ be a field of characteristic zero. A finitely generated $K$-algebra $R=K[x_1,\dots,x_m]/\fraka$ is of \emph{F-regular type} if there exists a finitely generated $\ZZ$-algebra $A\subseteq K$, and a finitely generated free $A$-algebra
\[
R_A=A[x_1,\dots,x_m]/\fraka_A\,,
\]
such that $R\cong R_A \otimes_A K$ and, for all maximal ideals $\mu$ in a Zariski dense subset of $\Spec A$, the fiber rings $R_A \otimes_A A/\mu$ are F-regular rings of characteristic $p>0$. Similarly, $R$ is of \emph{F-rational type} if for a dense subset of $\mu$, the fiber rings $R_A \otimes_A A/\mu$ are F-rational. Combining results from \cite{Hara,HW,MS,Smith:ratsing} one has:

\begin{theorem}
\label{thm:HSWMS}
Let $R$ be a ring which is finitely generated over a field of characteristic zero. Then $R$ has rational singularities if and only if it is of F-rational type. If $R$ is $\QQ$-Gorenstein, then it has log terminal singularities if and only if it is of F-regular type.
\end{theorem}

\begin{proposition}
\label{proposition:freg}
Let $K$ be a field of characteristic $p>0$, and $R$ an $\NN$-graded normal ring which is finitely generated over $R_0=K$. Let $\omega$ denote the graded canonical module of $R$, and set $d=\dim R$.

Suppose $R$ is F-regular. Then, for each integer $k$, there exists $q=p^e$ such that 
\[
\rank_KR_k\le\rank_K{[H^d_\frakm(\omega^{(q)})]}_k\,.
\]
\end{proposition}

\begin{proof}
If $d\le 1$, then $R$ is regular and the assertion is elementary. Assume $d\ge2$. Let $\xi\in[H^d_\frakm(\omega)]_0$ be an element which generates the socle of $H^d_\frakm(\omega)$. Since the map $\omega^{[q]}\to\omega^{(q)}$ is an isomorphism in codimension one, $F^e(\xi)$ may be viewed as an element of $H^d_\frakm(\omega^{(q)})$ as in \cite{KW:dim2}.

Fix an integer $k$. For each $e\in\NN$, set $V_e$ to be the kernel of the vector space homomorphism
\begin{equation}
\label{eqn:socle}
R_k\to{\big[H^d_\frakm\big(\omega^{({p^e})}\big)\big]}_k\,,\qquad\text{ where }c \mapsto cF^e(\xi)\,.
\end{equation}
If $cF^{e+1}(\xi)=0$, then $F(cF^e(\xi))=c^pF^{e+1}(\xi)=0$; since $R$ is F-pure, it follows that $cF^e(\xi)=0$. Consequently the vector spaces $V_e$ form a descending sequence 
\[
V_1\supseteq V_2\supseteq V_3\supseteq\cdots\,.
\]
The hypothesis that $R$ is F-regular implies $\bigcap_eV_e=0$. Since each $V_e$ has finite rank, $V_e=0$ for $e\gg0$. Hence the homomorphism~\eqref{eqn:socle} is injective for $e\gg0$.
\end{proof}

We next record tight closure properties of general $\NN$-graded hypersurfaces. The results for F-purity are essentially worked out in \cite{HRpure}.

\begin{theorem}
\label{thm:general-graded}
Let $A=K[x_1,\dots,x_m]$ be a polynomial ring over a field $K$ of positive characteristic. Let $d$ be a nonnegative integer, and set $M=\binom{d+m-1}{d}-1$. Consider the affine space $\AA_K^M$ parameterizing the degree $d$ forms in $A$ in which $x_1^d$ occurs with coefficient $1$. 

Let $U$ be the subset of $\AA_K^M$ corresponding to the forms $f$ for which $A/fA$ F-pure. Then $U$ is a Zariski open set, and it is nonempty if and only if $d\le m$.

Let $V$ be the set corresponding to forms $f$ for which $A/fA$ is F-regular. Then $V$ contains a nonempty Zariski open set if $d<m$, and is empty otherwise.
\end{theorem}

\begin{proof}
The set $U$ is Zariski open by \cite[page~156]{HRpure} and it is empty if $d>m$ by \cite[Proposition~5.18]{HRpure}. If $d\le m$, the square-free monomial $x_1 \cdots x_d$ defines an F-pure hypersurface $A/(x_1 \cdots x_d)$. A linear change of variables yields the polynomial
\[
f=x_1(x_1+x_2)\cdots(x_1+x_d)
\]
in which $x_1^d$ occurs with coefficient $1$. Hence $U$ is nonempty for $d\le m$.

If $d\ge m$, then $A/fA$ has $a$-invariant $d-m\ge 0$ so $A/fA$ is not F-regular. Suppose $d<m$. Consider the set $W\subseteq\AA_K^M$ parameterizing the forms $f$ for which $A/fA$ is F-pure and $(A/fA)_{\bar{x}_1}$ is regular; $W$ is a nonempty open subset of $\AA_K^M$. Let $f$ correspond to a point of $W$. The element $\bar{x}_1\in A/fA$ has a power which is a test element; since $A/fA$ is F-pure, it follows that $\bar{x}_1$ is a test element. Note that $\bar{x}_2,\dots,\bar{x}_m$ is a homogeneous system of parameters for $A/fA$ and that $\bar{x}_1^{d-1}$ generates the socle modulo $(\bar{x}_2,\dots,\bar{x}_m)$. Hence the ring $A/fA$ is F-regular if and only if there exists a power $q$ of the prime characteristic $p$ such that 
\[
x_1^{(d-1)q+1}\notin(x_2^q,\dots,x_m^q,f)A\,.
\]
The set of such $f$ corresponds to an open subset of $W$; it remains to verify that this subset is nonempty. For this, consider
\[
f=x_1^d+x_2\cdots x_{d+1}\,,
\]
which corresponds to a point of $W$, and note that $A/fA$ is F-regular since
\[
x_1^{(d-1)p+1}\notin(x_2^p,\dots,x_m^p,f)A\,.\qedhere
\]
\end{proof}

These ideas carry over to multi-graded hypersurfaces; we restrict below to the bigraded case. The set of forms in $K[x_1,\dots,x_m,y_1,\dots,y_n]$ of degree $(d,e)$ in which $x_1^dy_1^e$ occurs with coefficient $1$ is parametrized by the affine space $\AA_K^N$ where $N=\binom{d+m-1}{d}\binom{e+n-1}{e}-1$.

\begin{theorem}
\label{thm:general-bigraded}
Let $B=K[x_1,\dots,x_m,y_1,\dots,y_n]$ be a polynomial ring over a field $K$ of positive characteristic. Consider the $\NN^2$-grading on $B$ with $\deg x_i=(1,0)$ and $\deg y_j=(0,1)$. Let $d,e$ be nonnegative integers, and consider the affine space $\AA_K^N$ parameterizing forms of degree $(d,e)$ in which $x_1^dy_1^e$ occurs with coefficient $1$. 

Let $U$ be the subset of $\AA_K^N$ corresponding to forms $f$ for which $B/fB$ is F-pure. Then $U$ is a Zariski open set, and it is nonempty if and only if $d\le m$ and $e\le n$.

Let $V$ be the set corresponding to forms $f$ for which $B/fB$ is F-regular. Then $V$ contains a nonempty Zariski open set if $d<m$ and $e<n$, and is empty otherwise.
\end{theorem}

\begin{proof}
The argument for F-purity is similar to the proof of Theorem~\ref{thm:general-graded}; if $d\le m$ and $e\le n$, then the polynomial $x_1\cdots x_dy_1\cdots y_e$ defines an F-pure hypersurface.

If $B/fB$ is F-regular, then $\bsa(B/fB)<\bszero$ implies $d<m$ and $e<n$. Conversely, if $d<m$ and $e<n$, then there is a nonempty open set $W$ corresponding to forms $f$ for which the hypersurface $B/fB$ is F-pure and $(B/fB)_{\bar{x}_1\bar{y}_1}$ is regular. In this case, $\bar{x}_1\bar{y}_1\in B/fB$ is a test element. The socle modulo the parameter ideal $(x_1-y_1,x_2,\dots,x_m,y_2,\dots,y_n)B/fB$ is generated by $\bar{x}_1^{d+e-1}$, so $B/fB$ is F-regular if and only if there exists a power $q=p^e$ such that
\[
x_1^{(d+e-1)q+1}\notin(x_1^q-y_1^q,x_2^q,\dots,x_m^q,y_2^q,\dots,y_n^q,f)B\,.
\]
The subset of $W$ corresponding to such $f$ is open; it remains to verify that it is nonempty. For this, use $f=x_1^dy_1^e+x_2\cdots x_{d+1}y_2\cdots y_{e+1}$.
\end{proof}

\bibliographystyle{amsalpha}

\end{document}